\documentclass[12pt]{article}
\usepackage{amsmath,amsfonts,amssymb,amscd,verbatim,comment, amsthm}
\usepackage{fullpage}
\usepackage{color}
\usepackage{array}
\usepackage{enumerate}
\usepackage{mathtools}
\allowdisplaybreaks
\numberwithin{equation}{section}
\theoremstyle{plain}
 
\newtheorem{theorem}[equation]{Theorem}

\newtheorem{lemma}[equation]{Lemma}

\theoremstyle{definition}

\title{Upper and lower bounds on the size of $B_k[g]$ sets}
\author{Griffin Johnston\thanks{Department of Mathematics \& Statistics, Villanova University, Villanova, PA, U.S.A. $\{${\texttt{jjohns79}, \texttt{michael.tait}}$\}$\texttt{@villanova.edu}. The second author is partially supported by National Science Foundation grant DMS-2011553.} \and Michael Tait\footnotemark[1] \and Craig Timmons\thanks{Department of Mathematics and Statistics,  California State University Sacramento,  U.S.A. \texttt{craig.timmons@csus.edu}. Research is supported in part by Simons Foundation Grant \#359419}}
\date{\today}

\begin{document}

\maketitle
\begin{abstract}
   A subset $A$ of the integers is a $B_k[g]$ set if the number of multisets from $A$ that sum to any fixed integer is at most $g$. Let $F_{k,g}(n)$ denote the maximum size of a $B_k[g]$ set in $\{1,\dots, n\}$. In this paper we improve the best-known upper bounds on $F_{k,g}(n)$ for $g>1$ and $k$ large. When $g=1$ we match the best upper bound of Green with an improved error term. Additionally, we give a lower bound on $F_{k,g}(n)$ that matches a construction of Lindstr\"om while removing one of the hypotheses.
\end{abstract}

\section{Introduction}

We will denote the integers $\{1,2,\dots, n\}$ by $[n]$. Given natural numbers $k$ and $g$, a subset of $\mathbb{Z}$ is called a $B_k[g]$ set if for any $m$ there are at most $g$ multisets $\{x_1,\dots, x_k\}$ such that $x_1+\cdots +x_k = m$ and $x_i\in A$. Determining bounds on the maximum size of a $B_k[g]$ set in $[n]$ is a difficult and well-studied problem and is the focus of this paper. Let $F_{k,g}(n)$ be the maximum size of a $B_k[g]$ set in $[n]$. 

When $k=2$ and $g=1$, $B_2[1]$ sets are called {\em Sidon sets} and have been studied extensively since being introduced by Sidon \cite{Sidon} in the context of Fourier series, 
and then studied further by Erd\H{o}s and Tur\'an \cite{ET} from a combinatorial perspective. It is known \cite{BC, L} that $F_{2,1}(n) \sim n^{1/2}$, and determining whether or not $F_{2,1}(n) = n^{1/2} + O(1)$ is a 500 USD Erd\H{o}s problem \cite{E}. Very recently the error term was improved
by Balogh, F\"uredi, and Roy \cite{BFR} to $F_{2,1}(n) \leq n^{1/2} + 0.998n^{1/4}$ for sufficiently large $n$.  This represents the first improvement upon the error term in over 50 years.  

For other choices of $k$ and $g$, the asymptotic behavior of $F_{k,g}(n)$ has not been determined, yet 
there are upper and lower bounds that give the order of magnitude as a function of $n$.  
First we discuss upper bounds.

If $A$ is a $B_k[g]$ set contained in $[n]$, then each of the $\binom{|A| + k - 1}{k}$ multisets of size $k$ from $A$ determines a sum which is an integer in $[kn]$.  Each such integer
can appear as a sum at most $g$ times so $\binom{|A| + k - 1}{k} \leq gkn$.  This implies $F_{k,g}(n) \leq (gk!kn)^{1/k}$ which is known as the trivial bound.  
When $g=1$, a $B_k[g]$ set is often called a $B_k$ set. Nontrivial bounds on the size of a $B_k$ set were given by Jia \cite{Jia2} and Kolountzakis \cite{K} for $k$ 
even, and by Chen \cite{C} for $k$ odd.  These bounds show that 
\begin{equation}\label{JiaChen}
    F_{k,1}(n) \leq 
    \left(
    \left\lfloor \frac{k}{2}\right\rfloor !
     \left\lceil \frac{k}{2}\right\rceil !
     kn
    \right)^{1/k} + O_k(1).
\end{equation}
When $k$ is large, these bounds were improved by Green \cite{G}, who proved
\begin{equation}\label{green upper bound}
    F_{k,1}(n) \leq \left( \left\lceil\frac{k}{2}\right\rceil ! \left\lfloor \frac{ k}{2}\right\rfloor !
\sqrt{\frac{\pi k}{2}}(1+ \epsilon_k)n
\right)^{1/k}.
\end{equation}
It is noted (see page 379 of \cite{G}) that $\epsilon_k$ can be taken to be $O(k^{-1/8})$.

For $g>1$, Cilleruelo, Ruzsa, and Trujillo \cite{CRT} improved the 
trivial bound by proving  
\begin{equation}\label{eq CRT}
    F_{k,g}(n) \leq \left( \frac{k! kgn}{1 + \cos^k(\pi/k)}\right)^{1/k}.
\end{equation}
 Using an idea that the author's attribute to Alon, Cilleruelo and Jim{\'e}nez-Urroz \cite{CJ} showed 
 \begin{equation}\label{eq CJ}
     F_{k,g}(n) \leq \left( \sqrt{3k}k!gn\right)^{1/k}.
 \end{equation}
When $3\leq k\leq 6$, \eqref{eq CRT} is better and \eqref{eq CJ} is better for large $k$. Currently the best general upper bound, proved by the third author \cite{T}, is  
\begin{equation}\label{eq craig}
    F_{k,g}(n) \leq (1+o(1))\left(  \frac{x_k k! k gn}{\pi}\right)^{1/k}.
\end{equation}
Here $x_k$ is the unique real number in $(0,\pi)$ that satisfies $\frac{\sin x_k}{x_k} = \left( \frac{4}{3-\cos(\pi/k)} -1 \right)^k$. 
In \cite{T} it is shown that this upper bound is better than both \eqref{eq CRT} and \eqref{eq CJ} but that $\frac{x_k k! k g}{\pi} \to \sqrt{3k}k!g$ as $k\to\infty$.

Our first main theorem is an upper bound for large $k$ that improves \eqref{eq craig} and matches \eqref{green upper bound}. Our theorem improves the error term in \eqref{green upper bound} from $O(k^{-1/8})$ \cite{G} to $O(k^{-1/3})$.

\begin{theorem}\label{thm: normal upper bound}
Let $A \subset [n]$ be a $B_k[g]$ set. 
\begin{itemize}
    \item [(i)] When $g = 1$
    \[
|A| \leq \left( \left\lceil\frac{k}{2}\right \rceil !
\left\lfloor\frac{k}{2}\right\rfloor!
\sqrt{\frac{\pi k}{2}}(1+ O(k^{-1/3}))n
\right)^{\frac{1}{k}},
\]
and 
\item[(ii)] when $g>1$ 
\[
|A| \leq \left(\sqrt{\frac{\pi k}{2}}k!g(1+ O(k^{-1/3}))n\right)^{\frac{1}{k}}.
\]
\end{itemize}

\end{theorem}

The proof of Theorem \ref{thm: normal upper bound} uses the Berry-Esseen theorem and is inspired by recent work of Dubroff, Fox, and Xu \cite{DFX} who used a similar technique applied to the Erd\H{o}s distinct subset sums problem.

\medskip

Next we turn to lower bounds. Bose and Chowla \cite{BC} constructed a $B_k$ set of size $q$ in $\mathbb{Z}_{q^k-1}$ for $q$ an odd prime power, and this implies that $F_{k,1}(n) \geq (1-o(1))n^{1/k}$. Lindstr\"om \cite{L} showed
\begin{equation}\label{eq lindstrom}
F_{k,g}(n) \geq (1-o(1))(gn)^{1/k},
\end{equation}
when $g = m^{k-1}$ for some integer $m$. Cilleruelo and Jim\'enez-Urroz \cite{CJ} showed that for any $\epsilon > 0$, there exists a constant $g(k,\epsilon)$ so that for $g>g(k,\epsilon)$
\begin{equation}
    F_{k,g}(n) \geq \left((1-\epsilon) \sqrt{\frac{\pi k}{6} } g n \right)^{1/k}.
\end{equation}
Our second main theorem matches the bound in \eqref{eq lindstrom} and removes 
the requirement that $g = m^{k-1}$.  

\begin{theorem}\label{thm:lower bound}
For any integers $k \geq 2$ and $g \geq 1$, we have $F_{k,g}(n) \geq (1-o(1))(gn)^{1/k}$.
\end{theorem}

Before continuing the discussion, it is important to note that shortly after a preprint of this article was made available, the authors were notified by Carlos Trujillo that while not stated explicitly, Theorem \ref{thm:lower bound} follows from results of Caicedo, G\'{o}mez, and Trujillo \cite{CGT}.  The proof ideas are similar with the notable difference that \cite{CGT} works first in a general setting, and then specializes to known $B_h$-sets.  In our work we focus only on Bose-Chowla $B_h$ sets and consequently, the proof of the lower bound $F_{k,g}(n) \geq (1-o(1))(gn)^{1/k}$ is shorter.  However, we recommend \cite{CGT} for details on this technique and how it can be applied to $B_h$ sets of Bose and Chowla, Ruzsa, and G\'{o}mez and Trujillo. We leave in the details of the proof of Theorem \ref{thm:lower bound} for completeness and because it includes a density of primes argument that gives an asymptotic lower bound for all $n$; the theorem in \cite{CGT} applies only to an infinite sequence in $n$.

Previous work on bounding $F_{k,g}(n)$ is extensive and we have not included much of it. 
In particular, there are numerous papers that consider the case when $k=2$ and $g>1$ (e.g. \cite{CRV,  CV, K,MOB1, MOB2,  Yu1, Yu2}). 
For more information, see the surveys of Plagne \cite{P} and O'Bryant \cite{OB}. In Section \ref{sec upper bound} we prove Theorem \ref{thm: normal upper bound} and in Section \ref{sec lower bound} we prove Theorem \ref{thm:lower bound}.

\section{Upper bounds}\label{sec upper bound}

For a finite set $S\subset \mathbb{N}$, we define $\mathbb{E}(S) = \frac{1}{|S|}\sum_{s\in S} s$ and $\mathrm{Var}(S) = \frac{1}{|S|}\sum_{s\in S}(s - \mathbb{E}(S))^2$. For any random variable $X$ let $f_X$ be its probability distribution function and let $F_X$ be its cumulative distribution function. 

Assume that $A$ is a $B_k[g]$ set in $[n]$. To give an upper bound for the size of $A$, we will consider the distribution of sums of random elements of $A$. Define random variables $X_i$ that are independent and identically distributed by 
\[\mathbb{P}(X_i = a-\mathbb{E}(A)) = \frac{1}{|A|}\]
for every $a\in A$. Note that $\mathbb{E}(X_i) = 0$ and $\mathrm{Var}(X_i) = \mathrm{Var}(A)$. Define $\delta$ so that $\mathrm{Var}(A) = \delta n^2$. Any set of natural numbers up to $n$ has variance at most $\frac{(n-1)^2}{4}$ and hence $\delta \leq 1/4$. 

The details split into two cases depending on if $g=1$ or if $g>1$. When $g=1$ we take advantage of the fact that if $A$ is a $B_k[1]$ set, then for any $c\in \mathbb{Z}$ there is at most one solution (up to rearranging) to the equation
\begin{equation}\label{eq: unique differences}
a_1 + \cdots + a_{\lceil k/2 \rceil} - a_{\lceil k/2 \rceil + 1} - \cdots - a_k = c,
\end{equation}
where $a_1,\cdots, a_k \in A$.

When $g >1$, there are at most $g$ solutions (up to rearranging) to the equation 
\begin{equation}\label{eq: unique sums}
a_1+\dots +a_k =c.
\end{equation}

Define 
\begin{align*}
    Y_1 & = X_1 + \cdots + X_{\lceil k/2 \rceil}\\
    Y_2 & = X_{{\lceil k/2\rceil}+1} + \cdots + X_k\\
    Y &= Y_1-Y_2 \\
    Z & = X_1 + \cdots + X_k.
\end{align*}

When $g =1$ we will consider the random variable $Y$ and when $g>1$ we will consider the random variable $Z$. 

In \cite{CJ}, Cilleruelo and Jim\'enez-Urroz give an upper bound on the size of a $B_k[g]$ set for $g>1$ that depends on the variance of the set. Their proof is easily modified to include the $g=1$ case, and for our purposes it is more convenient to phrase the result in terms of the variance of the set. We give a short proof for completeness.

\begin{theorem}[Theorem 1.1 in \cite{CJ}]\label{variance upper bound}
Let $k,g\in \mathbb{N}$ be fixed. If $A$ is a $B_k[g]$ set in $[n]$, then for $g=1$ we have 
\[
 \frac{|A|^{2k}}{12\left(\lceil \frac{k}{2}\rceil !\lfloor \frac{k}{2}\rfloor!\right)^2} \leq (k+o(1))\mathrm{Var}(A),
\]
and for $g>1$ we have 
\[
 \frac{|A|^{2k}}{12(gk!)^2} \leq (k+o(1))\mathrm{Var}(A).
\]
\end{theorem}

\begin{proof}
For convenience we assume that $|A|^k$ is divisible by $\lceil \frac{k}{2}\rceil !\lfloor \frac{k}{2}\rfloor!$ when $g=1$ and by $gk!$ when $g>1$. If this is not the case we may truncate $A$ and let the $o(1)$ terms account for the difference.

The result follows by considering the variance of $Y$ when $g=1$ and the variance of $Z$ when $g>1$. Because the $X_i$ are independent, we have that $\mathrm{Var}(Y) = \mathrm{Var}(Z) = k\mathrm{Var}(A)$. To lower bound this quantity, the variance of $Y$ or $Z$ is as small as possible when the values taken by the random variables are as close together as possible. By \eqref{eq: unique differences} and \eqref{eq: unique sums}, when we look at all outputs of $Y$ or $Z$, each value can occur at most $\lceil \frac{k}{2}\rceil !\lfloor \frac{k}{2}\rfloor!$ times for $g=1$ and at most $gk!$ times for $g>1$. 

 Hence we have that the variance is bounded below by the variance of the multiset of integers from $1$ to $\ell$ where $\ell = \frac{|A|^k}{\lceil \frac{k}{2}\rceil !\lfloor \frac{k}{2}\rfloor!}$ when $g=1$, and where $\ell = \frac{|A|^k}{gk!}$ when $g>1$. Since each integer occurs the same number of times in this multiset, the variance of the multiset is the same as that of the discrete uniform distribution of integers up to $\ell$. This is given by 
 \[
 \mathrm{Var}(\{1,\cdots, \ell\}) = \frac{\ell^2-1}{12},
 \]
 and the result follows.
\end{proof}

When $k$ gets large, we can improve Theorem \ref{variance upper bound} by using more precise information about $Y$ and $Z$ than the variance. As $k$ goes to infinity, these distributions will be close to normal distributions, and we use the Berry-Esseen theorem to quantify this.

\begin{theorem} [Berry-Esseen Theorem] \label{berry esseen theorem}
Let $X_1,...,X_n$ be independent random variables with $\mathbb{E}[X_i] = 0$, $\mathbb{E}[X_i^2] = \mathrm{Var}(X_i)$ and 
$\mathbb{E}[|X_i|^3] = \rho_i < \infty$. 
Let $X = X_1 + \cdots + X_n$ and $\sigma^2 = \mathbb{E}[X^2]$. Then 
\[
\sup_{x\in \mathbb{R}} |F_X(x) - \Phi(x)| \leq .56\psi,
\]
where $F_X(x)$ and $\Phi(x)$ are the cumulative distribution functions for $X$ and the normal distribution with mean zero and standard deviation $\sigma$ respectively, and 
$\psi = \sigma^{-3}\cdot \sum_{i=1}^n \rho_i$.
\end{theorem}

By Theorem \ref{berry esseen theorem}, for any $j$ we can approximate $X_1 + \cdots + X_j$ by a normal random variable with mean $0$ and variance $j\delta n^2$ by using
\begin{align}\label{eq: berry esseen}
\begin{split}
\rho_i &= \mathbb{E}[|X_i|^3] \leq n\mathbb{E}[X_i^2] = \delta n^3,
\\
\sigma^2 &=\mathrm{Var}(X) = j\mathrm{Var}(X_i) = j\delta n^2,
\\
\psi &= \frac{1}{\sigma^3} \sum_{i=1}^j \rho_i \leq \frac{j\delta n^3}{j^{3/2}\delta^{3/2} n^3} = \frac{1}{\sqrt{j \delta}}.
\end{split}
\end{align}

We will approximate $Y_1$ by a normal distribution $\mathcal{N}(0, \lceil{\frac{k}{2}\rceil} \delta n^2)$ that has probability distribution $\varphi_1(x)$ and cumulative distribution function $\Phi_1(x)$. Similarly, let $\mathcal{N}(0, \lfloor \frac{k}{2}\rfloor \delta n^2)$ and $\mathcal{N}(0, k\delta n^2)$ have probability distribution functions $\varphi_2(x)$ and $\varphi(x)$ and cumulative distribution functions $\Phi_2(x)$ and $\Phi(x)$, respectively. Since $F_{Y_1}$ and $F_{Y_2}$ are close to $\Phi_1$ and $\Phi_2$ by Theorem \ref{berry esseen theorem}, we have that $F_Y$ is close to $\Phi$, quantified by the following lemma.

\begin{lemma}\label{lemma: berry esseen lemma}
    For $\Phi$ the cumulative distribution function of $\mathcal{N}(0, k \delta n^2)$, we have 
    \[
    \sup_x |F_Z(x) - \Phi(x)| \leq \frac{.56}{\sqrt{k\delta}},
    \]
    and
    \[
    \sup_x \left| F_Y(x) - \Phi(x)\right| \leq 4 \cdot \frac{.56}{\sqrt{\lfloor \frac{k}{2} \rfloor\delta}}.
    \]
\end{lemma}

\begin{proof}
The first inequality follows from Theorem \ref{berry esseen theorem} and \eqref{eq: berry esseen}. Now we prove the second. Since $Y = Y_1 - Y_2$ we have that $f_Y(x) = (f_{Y_1} \ast f_{-Y_2})(x)$. We also have $\varphi = \varphi_1 \ast \varphi_2$. Let $x$ be arbitrary. Then
\begin{align*}
    |F_Y(x) - \Phi(x)| = &\left| \int_{-\infty}^x f_{Y_1} \ast f_{-Y_2} - \varphi_1 * \varphi_2 \right| \\
    =& \left| \int_{-\infty}^x \varphi_2 \ast (f_{Y_1} - \varphi_1) + \varphi_1 \ast(f_{-Y_2} - \varphi_2) + (f_{Y_1} - \varphi_1)\ast (f_{-Y_2} - \varphi_2) \right| \\
    \leq & \left| \int_{-\infty}^x \varphi_2 \ast (f_{Y_1} - \varphi_1)  \right| +\left| \int_{-\infty}^x  \varphi_1 \ast(f_{-Y_2} - \varphi_2)\right| + \left| \int_{-\infty}^x (f_{Y_1} - \varphi_1)\ast (f_{-Y_2} - \varphi_2) \right|.
\end{align*}
We use Theorem \ref{berry esseen theorem} to show that each of the terms on the last line is small.

\begin{align*}
     \left| \int_{-\infty}^x \varphi_2 \ast (f_{Y_1} - \varphi_1) \right| 
    &= \left|\int_{-\infty}^x \int_{-\infty}^{\infty}
    \varphi_2(z-y)\Big(f_{Y_1}(y) - \varphi_1(y)\Big) dz dy \right| 
    \\
    &= 
    \left|\int_{-\infty}^x f_{Y_1}(y) - \varphi_1(y) \ dy \int_{-\infty}^{\infty}
    \varphi_2(z-y) dz \right|\\ 
       &= 
    \left|\int_{-\infty}^x f_{Y_1}(y) - \varphi_1(y) \ dy\right| \left|\int_{-\infty}^{\infty}
    \varphi_2(z-y) dz \right|\\ 
    &\leq \sup_{x\in \mathbb{R}}|F_{Y_1}(x) - \Phi_1(x)| 
    \left| \int_{-\infty}^\infty  \phi_2(z-y) dz \right| 
    \\ 
    &\leq \frac{.56}{\sqrt{\delta\lceil\frac{k}{2}\rceil}}
\end{align*}
where the last inequality follows because $\varphi_2$ is a probability distribution function and from \eqref{eq: berry esseen}. Similarly, noting that $\varphi_2$ is symmetric around $0$, we have that 
\[
\left| \int_{-\infty}^x  \varphi_1 \ast(f_{-Y_2} - \varphi_2)\right| \leq \frac{.56}{\sqrt{\delta\lfloor\frac{k}{2}\rfloor}},
\]
and 
\[
\left| \int_{-\infty}^x (f_{Y_1} - \varphi_1)\ast (f_{-Y_2} - \varphi_2) \right| \leq \left| \int_{-\infty}^x f_{Y_1} \ast (f_{-Y_2} - \varphi_2) \right| + \left| \int_{-\infty}^x \varphi_1\ast (f_{-Y_2} - \varphi_2) \right| \leq 2 \frac{.56}{\sqrt{\delta\lfloor\frac{k}{2}\rfloor}}
\]
\end{proof}

We now have everything that we need to prove Theorem \ref{thm: normal upper bound}.

\begin{proof}[Proof of Theorem \ref{thm: normal upper bound}]
Let $k$ and $g$ be fixed and assume that $A$ is a $B_k[g]$ set in $[n]$. As before, let $\mathrm{Var}(A) = \delta n^2$. If $\delta < \frac{\pi}{24}$, then we may apply Theorem \ref{variance upper bound} to obtain the claimed upper bound. Hence for the remainder of the proof we may assume that $\frac{\pi}{24}\leq \delta \leq \frac{1}{4}$. Since the $X_i$ are independent we have that the standard deviations of the random variables $Y$ and $Z$ are the same and we will denote this quantity by $\sigma$. We will consider the probability of the events that $-t < Y \leq t$ and $-t < Z \leq t$ where $t$ is an integer that will be chosen later. For $g=1$, by Lemma \ref{lemma: berry esseen lemma} and the assumption that $\delta \geq \pi/24$, we have that 
\[
|F_Y(x) - \Phi(x)| \leq \frac{4\cdot 0.56}{\sqrt{\lfloor \frac{k}{2} \rfloor \frac{\pi}{24}}},
\]
for all $x$. Hence

\begin{align*}
	\mathbb{P}[ -t < Y\leq t ] &= F_Y(t) - F_Y(-t) \\
	&= \Phi(t) - \Phi(-t) - 
	\Big( (F_Y(-t) - \Phi(-t)) - (F_Y(t) - \Phi(t))\Big) 
	\\
	&\geq 
	\Phi(t) - \Phi(-t) - 
	\Big|F_Y(-t) - \Phi(-t)\Big| - \Big|F_Y(t) - \Phi(t)\Big| \\
	&\geq \Phi(t) - \Phi(-t) - \frac{8\cdot 0.56}{\sqrt{\lfloor \frac{k}{2} \rfloor \frac{\pi}{24}}}
	\\
	&= 	\frac{1}{\sigma\sqrt{2\pi}} \int_{-t}^{t} \text{exp}(-\frac{x^2}{2\sigma^2})dx 
	-  \frac{4.48}{\sqrt{\lfloor \frac{k}{2} \rfloor \frac{\pi}{24}}}  \\
	&\geq \frac{(2t) \cdot \exp{( -\frac{t^2}{2\sigma^2})}}{\sigma\sqrt{2\pi}}
	-  \frac{4.48}{\sqrt{\lfloor \frac{k}{2} \rfloor \frac{\pi}{24}}}.
\end{align*}

On the other hand, by \eqref{eq: unique differences}, we have that for any fixed $x$
\[
	\mathbb{P}[Y = x] \leq \left(\left\lceil \frac{k}{2}\right\rceil!\right)\left(\left\lfloor \frac{k}{2} \right\rfloor!\right) |A|^{-k}.
\]
Combining these two inequalities yields
\[
 \frac{(2t) \cdot \exp{( -\frac{t^2}{2\sigma^2})}}{\sigma\sqrt{2\pi}} 
	- \frac{4.48}{\sqrt{\lfloor \frac{k}{2} \rfloor \frac{\pi}{24}}}
\leq \mathbb{P}[ -t<Y\leq t] \leq (2t)\left(\left\lceil \frac{k}{2}\right\rceil!\right)\left(\left\lfloor \frac{k}{2} \right\rfloor!\right) |A|^{-k}.
\]
Using the inequality $1-x\leq e^{-x}$ for all $x$ and $\sigma^2 \geq k\pi n^2/24$, we have
\[
e^{\frac{-t^2}{2\sigma^2}} \geq 1 - \frac{t^2}{2\sigma^2} \geq 1 - \frac{12t^2}{\pi k n^2}.
\]
Applying this inequality, dividing both sides by $2t$ and using $\sigma \leq \frac{\sqrt{k}n}{2}$ leads to
\[
\frac{1-\frac{12t^2}{\pi k n^2}}{n\sqrt{\pi k /2}} - \frac{4.48}{2t\sqrt{\lfloor \frac{k}{2} \rfloor \frac{\pi}{24}}} \leq \left(\left\lceil \frac{k}{2}\right\rceil!\right)\left(\left\lfloor \frac{k}{2} \right\rfloor!\right) |A|^{-k}.
\]

Setting $t = k^{1/3}n$, we find that s
\[
\frac{1 - \frac{12}{\pi k^{1/3}}}{n\sqrt{\pi k/2}} -   \frac{\frac{4.48}{2k^{1/3}n \sqrt{\lfloor \frac{k}{2} \rfloor \frac{\pi}{24}}} \cdot n\sqrt{\pi k/2}}{n\sqrt{\pi k/2}}\leq \left(\left\lceil \frac{k}{2}\right\rceil!\right)\left(\left\lfloor \frac{k}{2} \right\rfloor!\right) |A|^{-k}.
\]
Rearranging gives the result for $g=1$. For $g>1$ we use \eqref{eq: unique sums} and have that 
\[
\mathbb{P}[Z = x] \leq g k! |A|^{-k},
\]
for any $x$. Performing a similar calculation on $\mathbb{P}(-k^{1/3}n < Z \leq k^{1/3}n)$ gives the result, and we omit these details.

\end{proof}

\section{Lower Bounds}\label{sec lower bound}

In this section we prove Theorem \ref{thm:lower bound}.  The idea is to begin 
with a known construction of a $B_k[1]$ set and then consider the image of that set in a 
quotient group.  This idea has been used in other extremal graph theory and combinatorial number theory problems before \cite{ARS, BFR, F, livinsky, N, threepartite}.
In particular, and as noted in the discussion following Theorem \ref{thm:lower bound}, Caicedo, G\'{o}mez, and Trujillo \cite{CGT} contains a more general approach than what is done here.

Let $g$ and $k$ be fixed positive integers with $k \geq 2$. 
Assume that $q$ is a power of prime such that $g$ divides $q - 1$.  
Let $\mathbb{F}_{q^k}$ be the finite field with $q^k$ elements and 
let $\theta$ be a generator of the multiplicative group $\mathbb{F}_{q^k}^*$ of nonzero
elements of $\mathbb{F}_{q^k}$.  
Bose and Chowla \cite{BC} proved that 
\[
A = \{a \in \mathbb{Z}_{q^k -1} :\theta^a - \theta \in \mathbb{F}_q\}
\]
is a $B_k[1]$ set in $\mathbb{Z}_{q^k-1}$.  Let $\mu = \frac{ q^k-1}{g}$ 
and let $H$ be the subgroup of $\mathbb{Z}_{q^k - 1}$ generated by $\mu$.  Thus,
$H$ is the unique subgroup of $\mathbb{Z}_{q^k-1}$ with $g$ elements.  

Next we prove a lemma that is essential to the construction.  
This lemma is essentially known (Lemma 2.2 of \cite{TT} or Lemma 2.1 of \cite{BDT}).  
A short proof is included for completeness.

\begin{lemma}\label{AcapALemma}
    If $g$, $k$, $q$, $H$ and $A$ are given as above and $A - A : = \{ a - b : a,b \in A \}$, then 
     $(A-A) \cap H = \{0\}$.
\end{lemma}
\begin{proof}
Suppose $a,b \in A$ and $a - b \in H$.  There is an element 
$s \in \{0,1, \dots , g- 1 \}$ such that $a - b \equiv s \mu ( \textup{mod}~q^k - 1 )$.  
so $\theta^{ a- b} = \theta^{s \mu }$ in $\mathbb{F}_{q^k}$.  Let $\alpha , \beta \in \mathbb{F}_q$ 
satisfy $\theta^a = \theta + \alpha$ and $\theta^b = \theta + \beta$.  Observe that 
$ ( \theta^{ s \mu } )^{q- 1} = ( \theta^{q^k - 1} )^{ \frac{ s ( q - 1) }{g} } = 1$ 
and so $\theta^{ s \mu } \in \mathbb{F}_q$.  From $\theta^{a-b} = \theta^{s \mu}$, 
it follows that $\theta + \alpha = \theta^{s \mu } ( \theta + \beta)$ so 
\[
( \theta^{s \mu } - 1) \theta + \theta^{s \mu } \beta - \alpha = 0.
\]
The 
minimal polynomial of $\theta$ over $\mathbb{F}_q$ has degree $k \geq 2$ and 
so we must have $\theta^{ s \mu  }  - 1 = 0$ and $\theta^{s \mu } \beta - \alpha = 0$.  
In particular, the first equation implies $s = 0$ and so $a \equiv b( \textup{mod}~q^k - 1)$.  
\end{proof}

In the quotient group $\Gamma := \mathbb{Z}_{q^k - 1} / H$, let $A_H = \{ a +H : a \in A \}$.  
If $a  + H = b + H$ for some $a+H,b+H \in A_H$, then $a- b \in H$ which, 
by Lemma \ref{AcapALemma}, implies $a \equiv b ( \textup{mod}~q^k - 1 )$.  
Hence, $q = |A| = |A_H|$.  Next, we prove that $A_H$ is a 
$B_k [g]$ set in $\Gamma$.  Suppose $c + H \in \Gamma$ and 
we have 
\begin{equation}\label{eq:lb 1}
a_1 + H +  \dots + a_k + H = c +H 
\end{equation}
for some $a_i + H \in A_H$.  We will show that there are at most $g$ such 
solutions up to the ordering of the terms on the left hand side of (\ref{eq:lb 1}). 
Indeed, (\ref{eq:lb 1}) implies $a_1 + \dots + a_k \equiv c + h ( \textup{mod}~q^k - 1)$ for some 
$h \in H$.  There is at most one multiset $\{a_1, \dots a_k \}$ from $A$ that 
is a solution to this equation.  As there are $g$ choices for $h$, 
there will be at most $g$ multisets $\{a_1 + H , \dots , a_k + H \}$ from 
$A_H$ that are solutions to (\ref{eq:lb 1}).  Therefore, $A_H$ is a 
$B_k [g]$ set in $\Gamma$.  



We now finish the proof using a density of primes argument.  
For positive integers $x$, $c$, and $m$, let 
$\pi ( x ; c , m )$ be the number of primes $p$ for which $p \leq x$ and 
$p \equiv c (\textup{mod}~m)$. Writing $\phi$ for the Euler phi function, the 
prime number theorem in arithmetic progressions states that if $\textup{gcd}(c,m) =1$, then
\[
\pi ( x ; c , m ) = \frac{x}{ \phi ( m) \ln x} + O \left(  \frac{ x }{ \ln^2 x } \right).
\]
Let $\alpha = \lfloor (1 - \epsilon)^{1/k} (gn)^{1/k} \rfloor$ and
$\beta = \lfloor ( gn )^{1/k} \rfloor$.  We then have 
\begin{equation}\label{primes in ap}
\pi ( \beta ; 1 , g ) - \pi ( \alpha ; 1 , g ) 
\geq 
\frac{1}{ \phi (g) } \left(  \frac{ \beta}{ \ln \beta } - \frac{ \alpha }{ \ln \alpha } 
- O \left( \frac{ \beta }{ \ln^2 \beta } \right) \right).
\end{equation}
For large enough $n$ depending on $\epsilon$, $g$, and $k$, the 
right hand side of (\ref{primes in ap}) is positive since $\frac{ x }{ \log x}$ is 
strictly increasing for $x > e$.  Thus, there is a prime $q$ with 
$q \equiv 1 ( \textup{mod}~g)$ and $\alpha \leq q \leq \beta$.  We can 
then choose a $B_k[g]$-set $A$ in the group $\mathbb{Z}_{q^k-1} / H $ where
$|A| = q$.  This group is isomorphic to the cyclic group $\mathbb{Z}_{(q^k-1)/g }$ 
and we let $A'$ be a $B_k[g]$-set in this cyclic group.  Since $q \leq \beta$, we 
have $\frac{q^k-1}{g} \leq n$.  Therefore, we can view $A' $ as a subset of 
$\{1,2, \dots , n \}$ and under integer addition, $A'$
is still a $B_k[g]$-set.  It remains to show that $q \geq (1 - o(1))(gn)^{1/k}$, but
this follows from the definition of $\alpha$ and the inequality $q \geq \alpha$.  We conclude that 
for all positive integers $g$ and $k$ with $k \geq 2$, 
\[
F_{k,g}(n) \geq ( 1 - o(1)) ( gn)^{1/k}.
\]


\section*{Acknowledgments} 
The authors would like to thank Carlos Trujillo for bringing \cite{CGT} to our attention.
\bibliographystyle{plain}
\bibliography{bib.bib}

\end{document}